\documentclass[preprint]{amsart}       
\usepackage{graphicx}

\usepackage{amsfonts}
\usepackage{amscd}
\usepackage{amssymb}
\usepackage{amsmath}
\usepackage{stmaryrd}
\usepackage{color}
\usepackage{verbatim}
\usepackage{mathrsfs}
\usepackage{dsfont}
\usepackage{tikz}
\usepackage{subfigure}

\newtheorem{thm}{Theorem}[section]

\newtheorem{lem}[thm]{Lemma}

\newtheorem{cor}[thm]{Corollary}
\newtheorem{example}[thm]{Example}

\newtheorem{remark}[thm]{Remark}

\newcommand{\calI}{\mathcal{I}}
\newcommand{\calM}{\mathcal{M}}

\begin{document}

\title{Counting self-avoiding walks on free products of graphs
}


\author{Lorenz A.~Gilch}         
\address{Graz University of Technology, Institut f\"ur Mathematische Strukturtheorie, Steyrergasse 30, 8010 Graz, Austria}
\email{Lorenz.Gilch@freenet.de}

\author{Sebastian M\"uller}
\address{CNRS, Centrale Marseille, I2M, UMR 7373, 
113453 Marseille, France}
\email{sebastian.muller@univ-amu.fr}
\thanks{The research was supported by the exchange programme Amadeus-Amad\'ee $31473$TF}
\date{\today}


\begin{abstract}
The \textit{connective constant} $\mu(G)$ of a graph $G$ is the asymptotic growth rate of the number $\sigma_{n}$ of self-avoiding walks of length $n$ in $G$ from a given vertex. We prove a formula for the connective constant for free products of quasi-transitive graphs  and show that $\sigma_{n}\sim A_{G} \mu(G)^{n}$ for some constant $A_{G}$ that depends on $G$. In the case of finite products $\mu(G)$ can be calculated explicitly and is shown to be an algebraic number. 
\end{abstract}

 \maketitle

\section{Introduction}
\label{intro}
Let $G=(V,E,o)$ be a rooted graph with vertex set $V$, edge set $E$, and some distinguished
vertex $o$. The graphs considered in this note are locally finite, connected and simple, i.e.~no multiple edges between any pair of vertices. An (undirected) edge $e\in E$ with endpoints
$u,v\in V$ is noted as $e=\langle u, v\rangle$. Two vertices $u,v$ are called
\textit{adjacent} if there exists an edge $e\in E$ such that $e=\langle
u,v\rangle$; in this case we write $u\sim v$. The vertex $o\in V$ is called the root (or origin) of the graph.

A \textit{path} of length $n\in\mathbb{N}$ on a graph $G$ is a sequence of vertices $[v_0,v_1,\dots,v_n]$
such that $v_{i-1} \sim v_{i}$ for $i\in\{1,\dots,n\}$. 
An \textit{$n$-step self-avoiding walk} (SAW) on a rooted graph $G$ is a path of length $n$ starting in $o$ where no vertex appears more than once.
The \textit{connective constant} describes the asymptotic growth of the number $\sigma_{n}$ of self-avoiding walks of length $n$ on a given graph $G$; it is defined as
\begin{equation}
\mu(G)=\lim_{n\to\infty} \sigma_{n}^{1/n},
\end{equation}
provided the limit exists. In fact, if the graph is (quasi-)transitive the
existence of the limit is guaranteed by subadditivity of $\sigma_{n}$, see
Hammersley \cite{Ha:57}.  
In the case of a finite graph $\mu(G)$ equals trivially zero. In the following
we will just write $\mu=\mu(G)$.

Self-avoiding walks were originally introduced on Euclidean lattices as a model
for long polymer chains. But in recent years, the study on hyperbolic lattices,
see Madras and Wu \cite{MaWu:05}, Swierczak and Guttmann \cite{SwGu:96}, 
on one-dimensional lattices, see Alm and Janson \cite{AlJa:90}, and on general graphs and Cayley graphs, see Grimmett and Li \cite{GrLi:13b,GrLi:13,GrLi:14,GrLi:15,GrLi:15b}, has received increasing attention from physicists and mathematicians alike.
In particular, the connective constant of a $d$-regular vertex-transitive simple graph is shown to be bounded below by $\sqrt{d-1}$, see \cite{GrLi:15}.
 However, exact values of the connective constant are only known for   a small class of non-trivial graphs, namely  ladder graphs, see \cite{AlJa:90}, and the
 hexagonal lattice where the connective constant equals $\sqrt{2+\sqrt{2}}$,
 see Duminil-Copin and Smirnov \cite{DuSm:12}.
 
In the present paper we use generating functions, in particular the function $\calM (z) :=  \sum_{n\geq 0} \sigma_n\cdot z^n,$  to give formulas for the connective constant for the class of  free products of graphs.
Generating functions already played an important role in the theory of
self-avoiding walks, e.g. see  \cite{AlJa:90} and Bauerschmidt, Duminil-Copin,
Goodman and Slade \cite{BaDuGoSl:12}. The approach of this note is that the involved generating functions satisfy a functional equation of the same type as some generating functions for the ordinary random walk. As a consequence the connective constant can be expressed as the smallest positive root of an equation, see Theorem \ref{thm:connective constant}.    In the case of  products of finite graphs this equation can be solved explicitly and it is shown that $\mu$ is an algebraic number, see Corollary \ref{cor:algebraic}.

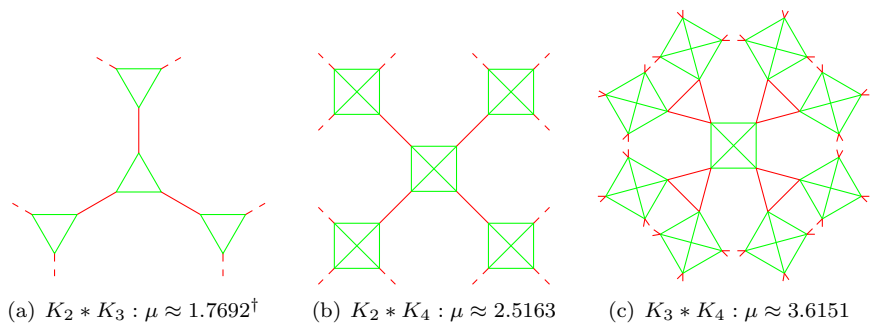
\begin{figure}[!h]
\centering
\subfigure[\mbox{$K_{2} * K_{3}: \mu\approx 1.7692^\dag$}]{
\begin{tikzpicture}[scale = 0.3]
\draw[green] (0,0) -- (2,0) -- (1,{sqrt(3)}) -- (0,0); 


\begin{scope}[shift={(2,0)},rotate=-30]
\draw[red] (0,0) -- (2,0) node (A){}; 
\end{scope}

\begin{scope}[shift={(0,0)},rotate=210]
\draw[red] (0,0) -- (2,0) node (B){}; 
\end{scope}

\begin{scope}[shift={(1,{sqrt(3)}) },rotate=90]
\draw[red] (0,0) -- (2,0);
\end{scope}


\begin{scope}[shift={(1,{2+sqrt(3)}) },rotate=60]
\draw[green] (0,0) -- (2,0) node (C1){}  -- (1,{sqrt(3)}) node (C2){}  -- (0,0); 
\end{scope}

\begin{scope}[shift={(A) },rotate=-60]
\draw[green] (0,0) -- (2,0) node (A1){} -- (1,{sqrt(3)}) node(A2){} -- (0,0); 
\end{scope}

\begin{scope}[shift={(B) },rotate=180]
\draw[green] (0,0) -- (2,0) node (B1){}  -- (1,{sqrt(3)}) node (B2){}  -- (0,0); 
\end{scope}


\begin{scope}[shift={(A1)},rotate=-90]
\draw[red, dashed] (0,0) -- (1,0) node (AA){}; 
\end{scope}

\begin{scope}[shift={(A2)},rotate=30]
\draw[red, dashed] (0,0) -- (1,0) node (AB){}; 
\end{scope}

\begin{scope}[shift={(B1)},rotate=150]
\draw[red, dashed] (0,0) -- (1,0) node (BA){}; 
\end{scope}

\begin{scope}[shift={(B2)},rotate=-90]
\draw[red, dashed] (0,0) -- (1,0) node (BB){}; 
\end{scope}

\begin{scope}[shift={(C1)},rotate=30]
\draw[red, dashed] (0,0) -- (1,0) node (CA){}; 
\end{scope}

\begin{scope}[shift={(C2)},rotate=150]
\draw[red, dashed] (0,0) -- (1,0) node (CB){}; 
\end{scope}
\end{tikzpicture}}
\quad
\subfigure[\mbox{$K_{2}\ast K_{4}: \mu\approx 2.5163$}]{

\begin{tikzpicture}[scale = 0.3]
\draw[green] (0,0) -- (2,0) -- (2,2) -- (0,2) -- (0,0); 
\draw[green] (0,0) -- (2,2);
\draw[green]  (2,0) -- (0,2); 


\begin{scope}[shift={(0,0)},rotate=-135]
\draw[red] (0,0) -- (2,0) node (A){}; 
\end{scope}

\begin{scope}[shift={(2,0)},rotate=-45]
\draw[red] (0,0) -- (2,0) node (B){}; 
\end{scope}

\begin{scope}[shift={(2,2) },rotate=45]
\draw[red] (0,0) -- (2,0) node (C){};
\end{scope}

\begin{scope}[shift={(0,2) },rotate=135]
\draw[red] (0,0) -- (2,0) node (D){};
\end{scope}


\begin{scope}[shift={(A) },rotate=180]
\draw[green] (0,0) -- (2,0) node (AB){} -- (2,2) node (AC){}-- (0,2) node (AD){} -- (0,0); 
\draw[green] (0,0) -- (2,2);
\draw[green]  (2,0) -- (0,2); 
\end{scope}

\begin{scope}[shift={(B) },rotate=-90]
\draw[green] (0,0) -- (2,0) node (BB){} -- (2,2) node (BC){}-- (0,2) node (BD){} -- (0,0); 
\draw[green] (0,0) -- (2,2);
\draw[green]  (2,0) -- (0,2); 
\end{scope}

\begin{scope}[shift={(C) },rotate=0]
\draw[green] (0,0) -- (2,0) node (CB){} -- (2,2) node (CC){}-- (0,2) node (CD){} -- (0,0); 
\draw[green] (0,0) -- (2,2);
\draw[green]  (2,0) -- (0,2); 
\end{scope}

\begin{scope}[shift={(D) },rotate=90]
\draw[green] (0,0) -- (2,0) node (DB){} -- (2,2) node (DC){}-- (0,2) node (DD){} -- (0,0); 
\draw[green] (0,0) -- (2,2);
\draw[green]  (2,0) -- (0,2); 
\end{scope}

\begin{scope}[shift={(AB)},rotate=135]
\draw[red, dashed] (0,0) -- (1,0) ; 
\end{scope}

\begin{scope}[shift={(AC)},rotate=-135]
\draw[red, dashed] (0,0) -- (1,0); 
\end{scope}

\begin{scope}[shift={(AD)},rotate=-45]
\draw[red, dashed] (0,0) -- (1,0); 
\end{scope}

\begin{scope}[shift={(BB)},rotate=-135]
\draw[red, dashed] (0,0) -- (1,0) ; 
\end{scope}

\begin{scope}[shift={(BC)},rotate=-45]
\draw[red, dashed] (0,0) -- (1,0); 
\end{scope}

\begin{scope}[shift={(BD)},rotate=45]
\draw[red, dashed] (0,0) -- (1,0); 
\end{scope}

\begin{scope}[shift={(CB)},rotate=-45]
\draw[red, dashed] (0,0) -- (1,0) ; 
\end{scope}

\begin{scope}[shift={(CC)},rotate=45]
\draw[red, dashed] (0,0) -- (1,0); 
\end{scope}

\begin{scope}[shift={(CD)},rotate=135]
\draw[red, dashed] (0,0) -- (1,0); 
\end{scope}

\begin{scope}[shift={(DB)},rotate=45]
\draw[red, dashed] (0,0) -- (1,0) ; 
\end{scope}

\begin{scope}[shift={(DC)},rotate=135]
\draw[red, dashed] (0,0) -- (1,0); 
\end{scope}

\begin{scope}[shift={(DD)},rotate=-135]
\draw[red, dashed] (0,0) -- (1,0); 
\end{scope}
\end{tikzpicture}

}\quad
\subfigure[${K}_{3}* K_{4}: \mu\approx 3.6151$]{

\begin{tikzpicture}[scale = 0.3]
\draw[green] (0,0) -- (2,0) -- (2,2) -- (0,2) -- (0,0); 
\draw[green] (0,0) -- (2,2);
\draw[green]  (2,0) -- (0,2); 


\begin{scope}[shift={(0,0)},rotate=-165]
\draw[red] (0,0) -- (2,0) node (A1){}  -- (1,{sqrt(3)}) node (A2){}  -- (0,0); 
\end{scope}

\begin{scope}[shift={(2,0)},rotate=-75]
\draw[red] (0,0) -- (2,0) node (B1){}  -- (1,{sqrt(3)}) node (B2){}  -- (0,0); 
\end{scope}

\begin{scope}[shift={(2,2)},rotate=15]
\draw[red] (0,0) -- (2,0) node (C1){}  -- (1,{sqrt(3)}) node (C2){}  -- (0,0); 
\end{scope}

\begin{scope}[shift={(0,2)},rotate=105]
\draw[red] (0,0) -- (2,0) node (D1){}  -- (1,{sqrt(3)}) node (D2){}  -- (0,0); 
\end{scope}


\begin{scope}[shift={(A1) },rotate=150]
\draw[green] (0,0) -- (2,0) node (A1B){} -- (2,2) node (A1C){}-- (0,2) node (A1D){} -- (0,0); 
\draw[green] (0,0) -- (2,2);
\draw[green]  (2,0) -- (0,2); 
\end{scope}

\begin{scope}[shift={(A2) },rotate=-150]
\draw[green] (0,0) -- (2,0) node (A2B){} -- (2,2) node (A2C){}-- (0,2) node (A2D){} -- (0,0); 
\draw[green] (0,0) -- (2,2);
\draw[green]  (2,0) -- (0,2); 
\end{scope}

\begin{scope}[shift={(D1) },rotate=60]
\draw[green] (0,0) -- (2,0) node (D1B){} -- (2,2) node (D1C){}-- (0,2) node (D1D){} -- (0,0); 
\draw[green] (0,0) -- (2,2);
\draw[green]  (2,0) -- (0,2); 
\end{scope}

\begin{scope}[shift={(D2) },rotate=120]
\draw[green] (0,0) -- (2,0) node (D2B){} -- (2,2) node (D2C){}-- (0,2) node (D2D){} -- (0,0); 
\draw[green] (0,0) -- (2,2);
\draw[green]  (2,0) -- (0,2); 
\end{scope}

\begin{scope}[shift={(B1) },rotate=-120]
\draw[green] (0,0) -- (2,0) node (B1B){} -- (2,2) node (B1C){}-- (0,2) node (B1D){} -- (0,0); 
\draw[green] (0,0) -- (2,2);
\draw[green]  (2,0) -- (0,2); 
\end{scope}

\begin{scope}[shift={(B2) },rotate=-60]
\draw[green] (0,0) -- (2,0) node (B2B){} -- (2,2) node (B2C){}-- (0,2) node (B2D){} -- (0,0); 
\draw[green] (0,0) -- (2,2);
\draw[green]  (2,0) -- (0,2); 
\end{scope}

\begin{scope}[shift={(C1) },rotate=-30]
\draw[green] (0,0) -- (2,0) node (C1B){} -- (2,2) node (C1C){}-- (0,2) node (C1D){} -- (0,0); 
\draw[green] (0,0) -- (2,2);
\draw[green]  (2,0) -- (0,2); 
\end{scope}

\begin{scope}[shift={(C2) },rotate=30]
\draw[green] (0,0) -- (2,0) node (C2B){} -- (2,2) node (C2C){}-- (0,2) node (C2D){} -- (0,0); 
\draw[green] (0,0) -- (2,2);
\draw[green]  (2,0) -- (0,2); 
\end{scope}


\begin{scope}[shift={(A1B) },rotate=90]
\draw[red, dashed] (0,0) -- ({1/2},0);
\draw[red, dashed] (0,0) -- ({1/sqrt(8)},{1/sqrt(8)});
\end{scope}

\begin{scope}[shift={(A1C) },rotate=180]
\draw[red, dashed] (0,0) -- ({1/2},0);
\draw[red, dashed] (0,0) -- ({1/sqrt(8)},{1/sqrt(8)});
\end{scope}

\begin{scope}[shift={(A1D) },rotate=-100]
\draw[red, dashed] (0,0) -- ({1/2},0);
\draw[red, dashed] (0,0) -- ({1/sqrt(8)},{1/sqrt(8)});
\end{scope}

\begin{scope}[shift={(A2B) },rotate=145]
\draw[red, dashed] (0,0) -- ({1/2},0);
\draw[red, dashed] (0,0) -- ({1/sqrt(8)},{1/sqrt(8)});
\end{scope}

\begin{scope}[shift={(A2C) },rotate=-135]
\draw[red, dashed] (0,0) -- ({1/2},0);
\draw[red, dashed] (0,0) -- ({1/sqrt(8)},{1/sqrt(8)});
\end{scope}

\begin{scope}[shift={(A2D) },rotate=-45]
\draw[red, dashed] (0,0) -- ({1/2},0);
\draw[red, dashed] (0,0) -- ({1/sqrt(8)},{1/sqrt(8)});
\end{scope}

\begin{scope}[shift={(B1B) },rotate=180]
\draw[red, dashed] (0,0) -- ({1/2},0);
\draw[red, dashed] (0,0) -- ({1/sqrt(8)},{1/sqrt(8)});
\end{scope}

\begin{scope}[shift={(B1C) },rotate=-100]
\draw[red, dashed] (0,0) -- ({1/2},0);
\draw[red, dashed] (0,0) -- ({1/sqrt(8)},{1/sqrt(8)});
\end{scope}

\begin{scope}[shift={(B1D) },rotate=0]
\draw[red, dashed] (0,0) -- ({1/2},0);
\draw[red, dashed] (0,0) -- ({1/sqrt(8)},{1/sqrt(8)});
\end{scope}

\begin{scope}[shift={(B2B) },rotate=-135]
\draw[red, dashed] (0,0) -- ({1/2},0);
\draw[red, dashed] (0,0) -- ({1/sqrt(8)},{1/sqrt(8)});
\end{scope}

\begin{scope}[shift={(B2C) },rotate=-45]
\draw[red, dashed] (0,0) -- ({1/2},0);
\draw[red, dashed] (0,0) -- ({1/sqrt(8)},{1/sqrt(8)});
\end{scope}

\begin{scope}[shift={(B2D) },rotate=60]
\draw[red, dashed] (0,0) -- ({1/2},0);
\draw[red, dashed] (0,0) -- ({1/sqrt(8)},{1/sqrt(8)});
\end{scope}

\begin{scope}[shift={(C1B) },rotate=-105]
\draw[red, dashed] (0,0) -- ({1/2},0);
\draw[red, dashed] (0,0) -- ({1/sqrt(8)},{1/sqrt(8)});
\end{scope}

\begin{scope}[shift={(C1C) },rotate=0]
\draw[red, dashed] (0,0) -- ({1/2},0);
\draw[red, dashed] (0,0) -- ({1/sqrt(8)},{1/sqrt(8)});
\end{scope}

\begin{scope}[shift={(C1D) },rotate=90]
\draw[red, dashed] (0,0) -- ({1/2},0);
\draw[red, dashed] (0,0) -- ({1/sqrt(8)},{1/sqrt(8)});
\end{scope}

\begin{scope}[shift={(C2B) },rotate=-45]
\draw[red, dashed] (0,0) -- ({1/2},0);
\draw[red, dashed] (0,0) -- ({1/sqrt(8)},{1/sqrt(8)});
\end{scope}

\begin{scope}[shift={(C2C) },rotate=60]
\draw[red, dashed] (0,0) -- ({1/2},0);
\draw[red, dashed] (0,0) -- ({1/sqrt(8)},{1/sqrt(8)});
\end{scope}

\begin{scope}[shift={(C2D) },rotate=135]
\draw[red, dashed] (0,0) -- ({1/2},0);
\draw[red, dashed] (0,0) -- ({1/sqrt(8)},{1/sqrt(8)});
\end{scope}

\begin{scope}[shift={(D1B) },rotate=0]
\draw[red, dashed] (0,0) -- ({1/2},0);
\draw[red, dashed] (0,0) -- ({1/sqrt(8)},{1/sqrt(8)});
\end{scope}

\begin{scope}[shift={(D1C) },rotate=90]
\draw[red, dashed] (0,0) -- ({1/2},0);
\draw[red, dashed] (0,0) -- ({1/sqrt(8)},{1/sqrt(8)});
\end{scope}

\begin{scope}[shift={(D1D) },rotate=180]
\draw[red, dashed] (0,0) -- ({1/2},0);
\draw[red, dashed] (0,0) -- ({1/sqrt(8)},{1/sqrt(8)});
\end{scope}

\begin{scope}[shift={(D2B) },rotate=45]
\draw[red, dashed] (0,0) -- ({1/2},0);
\draw[red, dashed] (0,0) -- ({1/sqrt(8)},{1/sqrt(8)});
\end{scope}

\begin{scope}[shift={(D2C) },rotate=150]
\draw[red, dashed] (0,0) -- ({1/2},0);
\draw[red, dashed] (0,0) -- ({1/sqrt(8)},{1/sqrt(8)});
\end{scope}

\begin{scope}[shift={(D2D) },rotate=-135]
\draw[red, dashed] (0,0) -- ({1/2},0);
\draw[red, dashed] (0,0) -- ({1/sqrt(8)},{1/sqrt(8)});
\end{scope}

\end{tikzpicture}}
\caption{Examples for free products of complete graphs $K_{n}$.}\label{fig:examples}
\end{figure}

\footnotetext{$^\dag$ The exact value of $\mu$ is $6 \Bigl(-2 + \sqrt[3]{46 - 6\sqrt{57}} + \sqrt[3]{46 + 6 \sqrt{57})}\Bigr)^{-1}$}

It is widely believed that on $\mathbb{Z}^{d}, d\neq 4,$ there is a critical exponent $\gamma$ depending on $d$ such that  $\sigma_{n}\sim A \mu^{n} n^{1-\gamma}$, where $A$ is a constant depending on $d$. However, this behaviour is rigorously proven only on lattices in dimensions $d\geq 5$, see Hara and Slade \cite{HaSl:92}, on ladder graphs, see \cite{AlJa:90}, and (in a weaker form) on regular tessellations of the hyperbolic lattice, see \cite{MaWu:05}. In this note we prove  that $\sigma_{n}\sim  A\mu(G)^{n} n^{1-\gamma}$ with $\gamma=1$ for free products of graphs, see Theorem \ref{thm:convergence-type}. This result supports the conjecture evoked in \cite{MaWu:05} that  SAW on  nonamenable graphs exhibits a mean-field behavior, i.e., that $\gamma=1$.

The study of stochastic processes on free products has a long and fruitful
history. In most of the works  generating function techniques played an
important role, e.g.~Woess \cite{Wo:00}. Similar techniques we use for rewriting
generating functions in terms of functions on the factors of the free product
were introduced independently and simultaneously in  Cartwright and Soardi
\cite{cartwright-soardi}, McLaughlin \cite{mclaughlin}, Voiculescu
\cite{voiculescu} and Woess \cite{woess3}. Free products owe some of their
importance Stalling's Splitting Theorem which states that a finitely generated
group has more than one (geometric) end if and only if it admits a nontrivial decomposition as an amalgamated free product or an HNN-extension over a finite subgroup. Furthermore, they constitute an important class of nonamenable graphs where calculations are still possible while they remain completely open on one-ended nonamenable graphs. For example, the spectral radius  of random walks, see \cite{Wo:00}, and the critical percolation probability $p_{c}$, see {\v{S}}pakulov{{\'a}} \cite{Ko:08,Sp:09}, can be calculated on free products but still constitute a big challenge on one-ended nonamenable graphs.

\section{Free products of graphs}

Let $r\in\mathbb{N}$ with $r\geq 2$ and set $\mathcal{I}:=\{1,\dots,r\}$.
Let $G_1=(V_1,E_1,o_1),\dots, G_r=(V_r,E_r,o_r)$ be a finite family of
undirected, connected, quasi-transitive, rooted
graphs with vertex sets $V_i$, edge sets $E_i$ and roots $o_i$ for $1\leq i\leq r$. 
We recall that a graph is called \textit{quasi-transitive} if its automorphism group acts quasi-transitive, i.e.~with finitely many orbits. In particular, every finite graph is quasi-transitive. Furthermore, we shall assume that we have $|V_i|\geq 2$ for every $i\in\mathcal{I}$ and that the vertex sets are distinct.
\par
Let $V_i^\times := V_i\setminus\{o_i\}$ for every $i\in\mathcal{I}$ and set
$\tau(x):=i$ if $x\in V_i^\times$. Define 
$$
V:=V_1\ast \dots \ast V_r =\{ x_1x_2\dots x_n \mid n\in\mathbb{N},   x_i\in
\bigcup_{j\in\mathcal{I}} V_j^\times,  \tau(x_i)\neq \tau(x_{i+1})\}\cup \{o\},
$$
which is the set of `words' over the alphabet $\bigcup_{i\in\mathcal{I}}V_i^\times $ such that no two consecutive letters come from the same
  $V_i^\times$. The empty word in $V$ is
  denoted by $o$. We extend the function $\tau$ on $V$ by setting $\tau(x_1\dots
  x_n):=\tau(x_n)$ for $x_1x_2\dots x_n\in V$. On the set $V$ we have a partial
  word composition law: if $x=x_1\dots x_m,y=y_1\dots y_n\in V$ with $\tau(x_m)\neq
  \tau(y_1)$ then $xy$ stands for the concatenation of $x$ and $y$, which is
  again an element of $V$. In particular, if $\tau(x)\neq i\in\mathcal{I}$ then
  we set $xo_i:=o_ix:=x$ and $xo:=ox:=x$. We regard
each $V_i$ as a subset of $V$, identifying each $o_i$ with $o$.
\par
We now equip the set $V$ with a
  graph structure by defining the set of edges $E$ as follows: if $i\in\mathcal{I}$
  and $x,y\in V_i$ with $x \sim y$ then $wx \sim wy$ for all $w\in V$
  with $\tau(w)\neq i$. The
\textit{free product} of the graphs
$G_1,\dots, G_r$ is then given by the graph 
$$
G:=G_1\ast G_2 \ast \dots \ast G_r := (V,E,o).
$$
In order to visualize this graph take a copy of each graph $G_1,\dots,G_r$ and
glue them together at their roots $o_1,\dots,o_r$, that is, we identify the
 roots of the single graphs as one common vertex in $V$, which becomes
 $o$. Inductively, at each element $x\in V$ with $\tau(x)=i\in\mathcal{I}$
 attach copies of the graphs $G_j$, $j\neq i$, where $x$ is identified with the roots $o_j$, $j\neq i$, of each single graph $G_j$.
\par
For $i\in\calI$ and $n\in\mathbb{N}$, let $\sigma_n^{(i)}$ be the number of self-avoiding walks of
length $n$ on $G_i$ starting at $o_i$. Since all graphs $G_i$ are
quasi-transitive there are, due to \cite{Ha:57},  numbers $\mu_i$ such that
$$
\mu_i=\lim_{n\to\infty} \bigl(\sigma_n^{(i)}\bigr)^{1/n}.
$$

\section{Functional equation for generating functions}
We turn to self-avoiding walks on the free product $G$.
Let $\sigma_n$  be the number of self-avoiding walks of
length $n\in\mathbb{N}$ on the free product $G$ starting at $o$. We set $\sigma_0:=1$.

For $i\in\calI$ and $z\in\mathbb{C}$, define the generating functions
\begin{eqnarray*}
\calM (z) &:= & \sum_{n\geq 0} \sigma_n\cdot z^n,\\
\calM_i (z) &:= & \sum_{n\geq 1} \sigma_n^{(i)}\cdot z^n.
\end{eqnarray*}
Furthermore, let $\bar\sigma_{n}^{(i)}$ be the number of self-avoiding walks of
length $n$ on the free product $G$ starting at $o$, which do visit $V_i^\times$.
Define
$$
\calM_i^\ast (z) := \sum_{n\geq 1} \bar\sigma_{n}^{(i)}\cdot z^n. 
$$
Due to the recursive structure of free products we can rewrite
\begin{eqnarray}
\calM_i^\ast (z) &=& \sum_{m\geq 0} \mathcal{M}_i(z) 
\sum_{\substack{i_1,\dots,i_m\in\mathcal{I}:\\ i_1\neq i,i_j\neq i_{j+1}}}\prod_{j=1}^m\mathcal{M}_{i_j}(z) \nonumber\\
&=& \mathcal{M}_i(z)\Bigl(1+\sum_{j\in\mathcal{I}\setminus\{i\}}
\mathcal{M}_j^\ast(z)\Bigr).\label{equ:M-i-ast}
\end{eqnarray}
Furthermore, we also have
\begin{equation}\label{equ:M}
\calM (z) = 1+ \sum_{i\in\mathcal{I}} \mathcal{M}_i^\ast(z).
\end{equation}
Plugging (\ref{equ:M}) into (\ref{equ:M-i-ast}) yields
$$
\mathcal{M}_i^\ast(z)= \mathcal{M}_i(z) \bigr(\mathcal{M}(z)-\mathcal{M}_i^\ast(z)\bigr),
$$
which in turn yields
$$
\mathcal{M}_i^\ast(z)=\mathcal{M}(z) \frac{\mathcal{M}_i(z)}{1+\mathcal{M}_i(z)}.
$$
Hence,
$$
\mathcal{M}(z)=1+\sum_{i\in\mathcal{I}}\mathcal{M}_i^\ast(z)= 1+ \mathcal{M}(z)\sum_{i\in\mathcal{I}} \frac{\mathcal{M}_i(z)}{1+\mathcal{M}_i(z)},
$$
or equivalently
\begin{equation}\label{equ:M-equation}
\mathcal{M}(z)=\frac{1}{1-\sum_{i\in\mathcal{I}} \frac{\mathcal{M}_i(z)}{1+\mathcal{M}_i(z)}}.
\end{equation}
The above equations do hold for every $z\in\mathbb{C}$ with $|z|<R(\mathcal{M})$ and $1+\mathcal{M}_i(z)\neq 0$, where
$R(\mathcal{M})$ is the radius of convergence of $\mathcal{M}(z)$.
\begin{remark}
Analogous functional equations as in (\ref{equ:M-equation}) appear in
different context for random walks on free groups, where one has some
generating functions on the single factors and deduces the corresponding
generating function on the free product. In Gilch \cite[Section
4]{Gi:07}, a functional equation of the same type was established in
order to compute the rate of escape of random walks on free products. In that
case the functions $\mathcal{M}_i(z)$ and $\mathcal{M}(z)$ are replaced by
probability generating functions. Furthermore, in Candellero, Gilch and
M\"uller \cite[Equation (4.1)]{CGM:12} a functional equation of the same type was also established in order to compute the Hausdorff dimension of branching random walks on free products; in this case the involved generating functions describe the evolution of particles. 
\end{remark}

We have  
$$
1\leq \limsup_{m\to\infty} \sigma_n^{1/n} = 1/R(\mathcal{M}),
$$ 
and therefore that $R(\mathcal{M})\leq 1$. In order to determine $R(\mathcal{M})$ 
we have to find -- by Prings\-heim's Theorem -- the smallest
singularity point on the positive $x$-axis of $\mathcal{M}(z)$. This
smallest singularity point is either one of  the radii of convergence
$R(\mathcal{M}_i)$ of the
functions $\mathcal{M}_i(z)$ or the smallest real positive number
$z_\ast$ with 
\begin{equation*}
\sum_{i\in\mathcal{I}}
\frac{\mathcal{M}_i(z_\ast)}{1+\mathcal{M}_i(z_\ast)}=1.
\end{equation*}
That is, $R(\mathcal{M}) = \min\bigl\lbrace
R(\mathcal{M}_1),\dots,R(\mathcal{M}_r),z_\ast \bigr\rbrace$. The next lemma
 answers which of the candidates is the radius of convergence of $\mathcal{M}(z)$:
\begin{lem}\label{lemma:z_ast}
$R(\mathcal{F})=z_\ast<\min \{R(\mathcal{M}_1),\dots,R(\mathcal{M}_r)\}$.
\end{lem}
\begin{proof}
If all generating functions $\mathcal{M}_i(z)$ have infinite radius of
convergence (that is, $R(\mathcal{M}_i)=\infty$), then $R(\mathcal{F})=z_\ast$ since $\mathcal{M}(0)=0$ and $\lim_{z\to\infty}
\mathcal{M}_i(z)=\infty$. Therefore, w.l.o.g.~we may assume from now on that $R(\mathcal{M}_1)<\infty$.
Recall that $\mu_1=\lim_{n\to\infty} (\sigma_n^{(1)})^{1/n} =
1/R(\mathcal{M}_1)$. Due to quasi-transitivity of the graph $G_1$ we
obtain with \cite[Theorem 1]{Ha:57} that
$$
\log \mu_1 =\inf_{n\in\mathbb{N}} \frac{\log \sigma_n^{(1)}}{n},
$$
%
that is, we have $ \sigma_n^{(1)}\geq \mu_{1}^{n}$.
\par
We follow a similiar reasoning as in \cite{CGM:12}.
We now consider $\sigma_n$ and count only those SAW which make at least one step
into the copy of $G_1$, before they make exactly one step into a copy of $G_2$ and then at least one step into
a copy of $G_1$, and so on. This leads to the following estimates for
$\sigma_n$:
\begin{eqnarray*}
\sigma_n & \geq & \sum_{k=1}^{\lfloor \frac{n}{2}\rfloor}
\sum_{\substack{n_1,\dots,n_k\geq 1:\\ n_1+\dots + n_k=n-k}}
\sigma_{n_1}^{(1)}\cdot \ldots \cdot \sigma_{n_k}^{(1)} \\
&\geq & \sum_{k=1}^{\lfloor \frac{n}{2}\rfloor}
\mu_1^{n-k} \binom{n-2k+k-1}{k-1}\\
&\geq & \sum_{k=0}^{\lfloor \frac{n}{2}\rfloor-1}
\mu_1^{n-k-1} \binom{n-k -2}{k}.
\end{eqnarray*}
The binomial coefficient in the second inequality arises as follows: we think
of placing $n-k$ balls into $k$ urns, where there should be at least
one ball in each urn. If $k\leq \lfloor \frac{n}{2}\rfloor-1$ then $n-k -2 \geq \lfloor
\frac{n}{2}\rfloor-1$. This yields together with the Binomial Theorem:
\begin{eqnarray*}
\sigma_n & \geq &  \mu_1^{n-1}\sum_{k=0}^{\lfloor
  \frac{n}{2}\rfloor -1}
\mu_1^{-k} \binom{\lfloor
\frac{n}{2}\rfloor-1}{k}\\
&=& \mu_1^{n-1} \bigl( 1+\mu_1^{-1}\bigr)^{\lfloor
\frac{n}{2}\rfloor-1}.
\end{eqnarray*}
This gives
$$
\liminf_{n\to\infty} \sigma_n^{1/n}  \geq  \liminf_{n\to\infty}
\mu_1^{\frac{n-1}{n}} \sqrt{1+\mu_1^{-1}} = \mu_1 \cdot \sqrt{1+\mu_1^{-1}}>
\mu_1.
$$
This finishes the proof, and the only remaining possibility for
$R(\mathcal{M})$ is $z_\ast$. 
\end{proof}

Eventually, we proved the following theorem.
\begin{thm}\label{thm:connective constant}
The connective constant $\mu=\mu(G)$ of a free product $G$ of quasi-transitive graphs is given as $\mu=1/z_{\ast}$, where $z_{\ast}$ is the smallest real positive number with 
\begin{equation*}
\sum_{i\in\mathcal{I}}
\frac{\mathcal{M}_i(z_\ast)}{1+\mathcal{M}_i(z_\ast)}=1.
\end{equation*}
\end{thm}
Finally, we consider the special case of free products of
\textit{finite} graphs.
If all single graphs $G_i$ are finite (that is, $V_i$ is finite), then
the generating functions $\mathcal{M}_i(z)$ are polynomials with positive
coefficients of degree less or
equal than $|V_i|-1$. We then have that $z_\ast$ is the smallest positive zero
of the polynomial
\begin{equation}\label{equ:denom-finite-case}
D(z)=\prod_{k\in\calI} \bigr(1+\calM_k(z)\bigr) - \sum_{i\in\calI} \calM_i(z)\cdot
\Bigl(\prod_{j\in\calI \setminus \{i\}} \bigr(1+\calM_j(z)\bigr) \Bigr).
\end{equation}
In the case $r=2$ we obtain that $z_\ast$ is the smallest positive zero
of the polynomial $1-\calM_1(z)\calM_2(z)$. This proves the following:
\begin{cor}\label{cor:algebraic}
If all factors $G_i$ are finite graphs, then $\mu$ is algebraic.
\end{cor}

\section{The critical exponent $\gamma$}

In this section we want to show that the critical exponent $\gamma=1$. We say that two functions $f(n),g(n)$ satisfy $f(n)\sim g(n)$, if $f(n)/g(n)\to 1$ as $n\to\infty$.
\begin{thm}\label{thm:convergence-type}
Let  $\sigma_{n}$ be the number of self-avoiding walks of length $n$ on a free product $G$ of quasi-transitive graphs and let $\mu(G)$ be the connective constant of $G$. Then, there exists some constant $A_{G}$ such that
\begin{equation*}
\sigma_n \sim A_{G}\mu(G)^n.
\end{equation*}
\end{thm}

In order to  prove the theorem we make some further observations.
We begin by recalling  that Equation (\ref{equ:M-equation}) is valid for all
$z\in\mathbb{C}$ with $|z|<z_\ast$ and $1+\mathcal{M}_i(z)\neq 0$ for all
$i\in\calI$. By multiplication of (\ref{equ:M-equation}) with $\prod_{i\in\calI}
\bigl(1+\mathcal{M}_i(z)\bigr)$ we can rewrite $\calM(z)$ as 
$$
\calM(z) = \frac{N(z)}{D(z)},
$$
where $N(z):=\prod_{i\in\calI} \bigl(1+\mathcal{M}_i(z)\bigr)$ and
$$
D(z):=\prod_{k\in\calI} \bigr(1+\calM_k(z)\bigr) - \sum_{i\in\calI} \calM_i(z)\cdot
      \prod_{j\in\calI \setminus \{i\}} \bigr(1+\calM_j(z)\bigr).
$$
\begin{lem}\label{lemma:N/D}
There exists some function $g(z)$ that is analytic in a neighbourhood of $z=z_\ast$  such that  $g(z_\ast)\neq 0$ and 
$$
\frac{N(z)}{D(z)}=\frac{g(z)}{z-z_\ast}.
$$
\end{lem}
\begin{proof}
Since $z_\ast < \min\{R(\calM_1),\dots,R(\calM_r)\}$ we can expand each
$\calM_i(z)$ around $z_\ast$ as a Taylor series of the form
$$
\calM_i(z) = \sum_{n\geq 0} a_n^{(i)} (z-z_\ast)^n.
$$
In particular, $N(z)$ and $D(z)$ are analytic in a neighbourhood of $z=z_\ast$.
Plugging this expansion into $D(z)$ we must have due to $D(z_\ast)=0$ that
$D(z)$ can be rewritten as 
\begin{equation}\label{equ:denom}
D(z)=\sum_{n\geq 1} d_n (z-z_\ast)^n.
\end{equation}
We now claim that $z_\ast$ is a single zero of
(\ref{equ:denom}), that is, we have $d_1\neq 0$: the derivative of (\ref{equ:denom})
is
\begin{eqnarray*}
&&\sum_{i\in\calI} \calM_i'(z) \prod_{k\in\calI\setminus\{i\}}
\bigr(1+\calM_k(z)\bigr) - \sum_{i\in\calI} \calM_i'(z)\cdot
\prod_{j\in\calI \setminus \{i\}} \bigr(1+\calM_j(z)\bigr)\\
&&\quad\quad
- \sum_{i\in\calI} \calM_i(z) \sum_{k\in\calI\setminus\{i\}} \calM_k'(z)
\prod_{l\in\calI \setminus \{i,k\}} \bigr(1+\calM_l(z)\bigr)\\
&=&- \sum_{i\in\calI} \calM_i(z) \sum_{k\in\calI\setminus\{i\}} \calM_k'(z)
\prod_{l\in\calI \setminus \{i,k\}} \bigr(1+\calM_l(z)\bigr).
\end{eqnarray*}
Since $\calM_i(z_\ast),\calM_i'(z_\ast)>0$ for all $i\in\calI$ the derivative
is strictly negative for real $z>0$. Consequently, $D'(z_\ast)\neq 0$ and therefore $z_\ast$
is a single zero of $D(z)$. Since $z_\ast$ is the smallest positive zero of $D(z)$ we
can write $D(z)=(z-z_\ast) \cdot d(z)$, where $d(z)$ is analytic in a
neighbourhood of $z=z_\ast$ with $d(z_\ast)\neq 0$. Setting $g(z):=N(z)\cdot
d(z)^{-1}$ yields the claim since $g(z)$ is a function with $g(z_\ast)\neq 0$, which is analytic in a neighbourhood of
$z_\ast$. \qed
\end{proof}
We now claim that $z_\ast$ is the unique dominant singularity of
$N(z)/D(z)$. 
\begin{lem}\label{lemma:no-further-singularity}
The only singularity of $N(z)/D(z)$ on its circle of convergence (of radius $z_\ast$) is  $z_\ast$.
\end{lem}
\begin{proof}
Expanding $D(z)$ yields that we can rewrite $D(z)=1-S(z)$, where $S(z)$ has
the form
$$
S(z)=\sum_{k=2}^r \sum_{1=i_1<i_2<\dots <i_k=r} a_{i_1,\dots,i_k} \mathcal{M}_{i_1}(z)\dots
\mathcal{M}_{i_k}(z)=\sum_{n\geq 1} s_n z^n,
$$
where $a_{i_1,\dots,i_k}\geq 0$ and $s_n>0$ for all $n\in\mathbb{N}$. In
particular, $S(z_\ast)=1$. It is sufficient to show that there is no
$z_0\in\mathbb{C}$ with $|z_0|=z_\ast$ and $S(z_0)=1$. For this purpose,
let be $\theta\in[0,2\pi)$ and consider
$$
1-S(e^{i\theta} z_\ast) = 1- \sum_{n\geq 1} s_n z_\ast^n e^{i\theta n}.
$$
Observe that the elements $e^{i\theta n}$ are points on the (complex) unit circle and
that $S(e^{i\theta} z_\ast)$ is a convex combination of these points since $S(z_\ast)=1$. Since $1$
is an extremal point of the convex unit disc and since the greatest common
divisor of the $s_n$'s equals $1$, we must have that $\theta=0$, which yields
the claim. 
\end{proof}
Now we can proceed with the proof of Theorem \ref{thm:convergence-type} which
uses the modern tool of Singularity Analysis developed by Flajolet and
Sedgewick \cite[Chapter VI]{FlSe:09}.
\begin{proof}[Proof of Theorem \ref{thm:convergence-type}.]
Recall that the formula for $\mathcal{M}(z)$ given by Equation
(\ref{equ:M-equation}) is only valid for all $z\in\mathbb{C}$ with
$|z|<z_\ast$, while the quotient $N(z)/D(z)$ is analytic in some domain of the
form
$$
\Delta=\{z\in\mathbb{C} \mid |z|<R', z\neq z_\ast, |\mathrm{arg}(z-z_\ast)|>\varphi\}
$$
for some $R'>z_\ast$ and some $\varphi\in (0,\pi/2)$ due to Lemma
\ref{lemma:no-further-singularity}; compare also with 
\cite[Definition VI.1]{FlSe:09}. Thus, $N(z)/D(z)$ is an analytic
continuation of $\mathcal{M}(z)$ to the domain $\Delta$. This is because, if
there is some $z_0\in\mathbb{C}\setminus \{z_\ast\}$ with $|z_0|=z_\ast$ and
$1+\mathcal{M}_i(z_0)=0$ for some $i\in\calI$ then $\mathcal{M}(z)\to 0$ as $z\to
z_0$, and by Riemann's theorem on removable singularities we obtain that
$N(z)/D(z)$ is an analytic representation of $\mathcal{M}(z)$ also in a
neighbourhood of $z=z_0$, yielding that $z_0$ is not singular.
%
%
In view of Lemma \ref{lemma:no-further-singularity} and the above discussion $z_\ast$ is the only
singularity of $\mathcal{M}(z)$ on the circle with radius $z_\ast$. With this
observation and with Lemma \ref{lemma:N/D} 
we can rewrite $N(z)/D(z)$ as
$$
\mathcal{M}(z) = \frac{a_0}{z-z_\ast} + \sum_{n\geq 1} a_n (z-z_\ast)^{n-1},
$$
where $g(z)=\sum_{n\geq 0} a_n (z-z_\ast)^n$. 
 By  \cite[Chapter VI.2]{FlSe:09}, the singular term \mbox{$(z-z_\ast)^{-1}$} leads to the
proposed form for the coefficients $\sigma_n$. 
\end{proof}

\section{Examples}

In this section we collect some examples with explicit calculations for free
products of \textit{finite} and \textit{infinite} graphs. Let $K_{n}$ be the complete graph of $n$ vertices.

\begin{example}
Consider the free product of $K_{2} \ast K_{3}$, see Figure \ref{fig:examples}(a).  Then 
$$
\calM_1(z)=z, \quad \calM_2(z)=2z+2z^2.
$$
The smallest positive zero of $1-\calM_1(z)\calM_2(z)=1-2z^2-2z^3$ is given by
$$
z_\ast=\frac{1}{6} \biggl(-2 + \sqrt[3]{46 - 6\sqrt{57}} + \sqrt[3]{46 + 6
  \sqrt{57})}\biggr)\approx 0.565198.
$$
Eventually, $\mu=1/z_\ast \approx 1.76929$.
\end{example}
\begin{example}
Consider $K_{2} \ast K_{3}\ast K_4$. Then 
$$
\calM_1(z)=z,\quad \calM_2(z)=2z+2z^2, \quad \calM_3(z)=3z+6z^2+6z^3.
$$
The smallest positive zero of the polynomial
$$
\prod_{k=1}^3 \bigr(1+\calM_k(z)\bigr) - \sum_{i=1}^3 \calM_i(z)\cdot
\Bigl(\prod_{ j\neq i} \bigr(1+\calM_j(z)\bigr) \Bigr).
$$
is given by $z_\ast \approx 0.210631$, which is the smallest positive zero
 of the polynomial $24z^6+60z^5+66z^4+38z^3+11z^2-1$.
That is, $\mu=1/z_\ast \approx 4.74763$.
\end{example}

\begin{example}
Let $C_{n}$ be the cycle graph with $n$ vertices.
Consider free products of the form $C_{2}\ast
C_{n}$ for $n\geq 3$. Then $\mathcal{M}_1(z)=z$ and
$\mathcal{M}_2(z)=\sum_{k=1}^{n-1} 2z^k$. The following table shows the approximations
of $\mu$ for different $n$:
$$
\begin{array}{|c||c|c|c|c|c|c|c|c|}
\hline n  &4 &5 &6 &7 &8  &9 & 10 &\infty \\
\hline
\mu  & 1.89932 & 1.95350 & 1.97781 & 1.98920 & 1.99468 & 1.99737 & 1.99869 & 2 \\
\hline
\end{array}
$$
Here, we set $C_{\infty}$ to be the one-dimensional Euclidean lattice. The free product $C_{2}\ast C_{\infty}$ is then the regular tree of degree $3$ and hence  $\mu=2$.
\end{example}

\begin{example}
Consider the following infinite graph $G_1$:\\

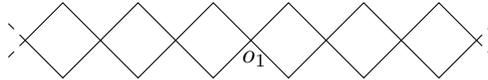
\begin{center}
\begin{tikzpicture}[scale = 0.5]
\begin{scope}[shift={(0,0)},rotate=0]
\draw[black] (0,0) -- (1,1) -- (2,0) -- (1,-1) -- (0,0); 
\end{scope}
\begin{scope}[shift={(2,0)},rotate=0]
\draw[black] (0,0) -- (1,1) -- (2,0) -- (1,-1) -- (0,0); 
\end{scope}
\begin{scope}[shift={(4,0)},rotate=0]
\draw[black] (0,0) -- (1,1) -- (2,0) -- (1,-1) -- (0,0); 
\end{scope}
\begin{scope}[shift={(6,0)},rotate=0]
\draw[black] (0,0) -- (1,1) -- (2,0) -- (1,-1) -- (0,0); 
\end{scope}
\node (A) at (6.1,-0.5) {$o_{1}$}; 
\begin{scope}[shift={(8,0)},rotate=0]
\draw[black] (0,0) -- (1,1) -- (2,0) -- (1,-1) -- (0,0); 
\end{scope}
\begin{scope}[shift={(10,0)},rotate=0]
\draw[black] (0,0) -- (1,1) -- (2,0) -- (1,-1) -- (0,0); 
\end{scope}
\begin{scope}[shift={(12,0)},rotate=0]
\draw[black, dashed] (0,0) -- (0.5,0.5); 
\draw[black, dashed] (0,0) -- (0.5,-0.5); 
\end{scope}
\begin{scope}[shift={(0,0)},rotate=180]
\draw[black, dashed] (0,0) -- (0.5,0.5); 
\draw[black, dashed] (0,0) -- (0.5,-0.5); 
\end{scope}

\end{tikzpicture}
\end{center} We have $\sigma_1^{(1)}=4$. Let us consider SAW whose first step is to the right. Every SAW of even length   consists only of steps which pass through an edge from the left to the right. At every crossing we have two possibilities to extend some given SAW towards the right hand side. For SAW of odd length each SAW of length $n-1$ can be extended to the right by two ways, and by one edge back to the left.
By symmetry this yields $\sigma_1^{(2n)}=2\cdot 2^{n}$ and  $\sigma_1^{(2n+1)}=2\cdot 3\cdot 2^{n}$ for $n\geq 1$. Hence,
\begin{eqnarray*}
\mathcal{M}_1(z)= \sum_{n\geq 1} \sigma_1^{(n)}z^n &=& 4z+\sum_{n\geq 1} 2^{n+1}z^{2n} +
\sum_{n\geq 1} 3\cdot 2^{n+1}\cdot z^{2n+1} \\
&=& \frac{2+6z}{1-2z^2}-2z-2.
\end{eqnarray*}
That is, $R(\mathcal{M}_1)=\frac{1}{\sqrt{2}}$. Consider again the complete graph of
degree $4$ denoted by $K_4$ with
$$
\calM_2(z)=3z+6z^2+6z^3.
$$
The connective constant of the free product $G_1\ast K_4$ is then given as the solution
$z_\ast$ of $D(z)=0$ which can be calculated numerically:
$$
z_\ast = 0.203143.
$$
The connective constant is then given by
$$
\mu = \frac{1}{z_\ast} = 4.92264.
$$
\end{example}


%
%
%


\bibliographystyle{spmpsci}      
\bibliography{saw}   
\end{document}